\documentclass[11pt]{amsart}
\usepackage{color}
\usepackage[bookmarks,colorlinks]{hyperref}
\usepackage{amssymb}

\theoremstyle{plain}
\newtheorem{thm}{Theorem}

\newtheorem{lem}[thm]{Lemma}
\newtheorem{prop}[thm]{Proposition}
\newtheorem{cor}[thm]{Corollary}

\theoremstyle{definition}

\theoremstyle{remark}
\newtheorem*{rem*}{Remark}
\newtheorem{rem}{Remark}

\newtheorem*{ack}{Acknowledgement}

\newcommand{\Rd}{{\mathbb{R}^d}}

\newcommand{\N}{{\mathbb{N}}}

\newcommand{\R}{\mathbb{R}}

\newcommand{\EEE}{\mathcal{E}}

\DeclareMathOperator{\WLSC}{WLSC}
\DeclareMathOperator{\WUSC}{WUSC}

\newcommand{\lc}{{\underline{c}}}
\newcommand{\uc}{{\overline{c}}}
\newcommand{\la}{{\underline{\alpha}}}
\newcommand{\ua}{{\overline{\alpha}}}

\newcommand{\lC}{{\underline{C}}}
\newcommand{\uC}{{\overline{C}}}
\newcommand{\lA}{{\underline{A}}}
\newcommand{\uA}{{\overline{A}}}

\definecolor{kb}{rgb}{0,0.6,0}
\definecolor{pk}{rgb}{0,0,0.6}
\definecolor{bd}{rgb} {0.6,0,0}

\date{\today}

\begin{document}
\title[Hardy inequalities for semigroups]{Hardy inequalities and non-explosion results for semigroups}

\author[K{.} Bogdan]{Krzysztof Bogdan}\address{Department of Mathematics, Wroc{\l}aw University of Technology,
Wybrze{\.z}e Wyspia{\'n}skiego 27,
50-370 Wroc{\l}aw, Poland}
\email{bogdan@pwr.edu.pl}
\author[B{.} Dyda]{Bart{\l}omiej Dyda}
\address{Department of Mathematics, Wroc{\l}aw University of Technology,
Wybrze{\.z}e Wyspia{\'n}skiego 27,
50-370 Wroc{\l}aw, Poland}
\email{Bartlomiej.Dyda@pwr.edu.pl}
\author[P{.} Kim]{Panki Kim}
\address{Department of Mathematics, Seoul National University,
Building 27, 1 Gwanak-ro, Gwanak-gu
Seoul 151-747, Republic of Korea}
\email{pkim@snu.ac.kr}


 %


\thanks{Krzysztof Bogdan and Bart\l{}omiej Dyda were partially supported by NCN grant 2012/07/B/ST1/03356.}

\begin{abstract}
We prove non-explosion results for Schr\"odinger
perturbations of symmetric transition densities and 
Hardy inequalities for their quadratic forms
by using explicit supermedian functions of their semigroups.
\end{abstract}
\maketitle

\section{Introduction}\label{sec:int}
Hardy-type inequalities are important in harmonic analysis, potential theory, functional analysis, partial differential equations and probability. In PDEs they are used to obtain a priori estimates, existence and regularity results \cite{MR2777530} and to study qualitative properties and asymptotic behaviour of solutions \cite{MR1760280}. 
In functional and harmonic analysis they yield embedding theorems and interpolation theorems, e.g. Gagliardo--Nirenberg interpolation inequalities \cite{MR2852869}.
The connection of Hardy-type inequalities to the theory of superharmonic functions in analytic and probabilistic potential theory was studied, e.g., in 
\cite{MR856511}, \cite{MR1786080}, \cite{MR2663757}, \cite{Dyda12}.
A general rule stemming from the work of P.~Fitzsimmons \cite{MR1786080} may be summarized as follows: if $\mathcal{L}$ is the generator of a symmetric Dirichlet form $\EEE$ and $h$ is superharmonic, i.e. $h\ge 0$ and $\mathcal{L}h\le 0$, then
$\EEE(u,u)\ge \int u^2 (-\mathcal{L}h/h)$. 
The present paper 
gives an analogous connection
in the setting of symmetric transition densities.
When these are integrated against increasing weights in time and arbitrary weights in space, we obtain suitable (supermedian) functions $h$.
The resulting analogues $q$ of the Fitzsimmons' ratio $-\mathcal{L}h/h$ yield explicit Hardy inequalities which  in many cases are optimal.
The approach is very general and the resulting Hardy inequality is automatically valid on the whole of the ambient 
$L^2$ space.

We also prove non-explosion results for Schr\"odinger perturbations of the original transition densities by the ratio $q$, namely we verify that $h$ is supermedian, in particular integrable with respect to the perturbation.
For instance we recover the famous non-explosion result of Baras and Goldstein  for $\Delta+(d/2-1)^2|x|^{-2}$, cf. \cite{MR742415} and \cite{MR3020137}. 

The results are illustrated by applications to transition densities with certain scaling properties.
 
The structure of the paper is as follows.
In Theorem~\ref{thm:perturbation} of Section~\ref{sec:Sc} we prove the non-explosion result for Schr\"odinger perturbations. 
In Theorem~\ref{thm:Hardy} of Section~\ref{sec:Hi} we prove the Hardy inequality. In fact, under mild additional assumptions we have a Hardy {\it equality} with an explicit remainder term.
Sections~\ref{sec:A}, \ref{s:asjm} and \ref{sec:wls} present applications.
In Section~\ref{sec:A} we 
recover
 the classical Hardy equalities for the quadratic forms of the Laplacian and fractional Laplacian. For completeness we also recover the best constants and the corresponding remainder terms, as given by Filippas and Tertikas \cite{MR1918494}, Frank, Lieb and Seiringer \cite{MR2425175}
 and Frank and Seiringer 
\cite{MR2469027}.
In Section~\ref{s:asjm} we consider transition densities with weak global scaling in the setting of metric spaces. These include a class of transition densities on fractal sets 
(Theorem~\ref{t:hardyphi1} and Corollary~\ref{c:hardyphi1})
and transition densities of many unimodal L\'evy processes on $\Rd$ (Corollary~\ref{cor:hardyphi}).
We prove Hardy inequalities for their quadratic forms. 
In Section~\ref{sec:wls} we focus on transition densities 
with weak local scaling on $\Rd$. The corresponding Hardy inequality is stated in Theorem~\ref{cor:hardyphi2}.

The calculations in 
Sections~\ref{sec:A}, \ref{s:asjm} and \ref{sec:wls}, which
produce explicit weights in Hardy inequalities, also give non-explosion results for specific Schr\"odinger perturbations of the corresponding transition densities by means of
Theorem~\ref{thm:perturbation}.
These are stated in 
Corollary~\ref{cor:aSp} and \ref{cor:aSpL} and Remark~\ref{rem:ef}, \ref{rem:efL} and \ref{rem:efLl}.

Currently our methods are confined to the (bilinear) $L^2$ setting. We refer to \cite{MR2469027}, \cite{MR3117146} for other frameworks.
Regarding further development, it is of interest to find relevant applications with less space homogeneity and scaling than required in the examples presented below, extend the class of considered time weights, prove explosion results for ``supercritical'' Schr\"odinger perturbations, and understand more completely when equality holds in our Hardy inequalities.

Below we use ``$:=$" to indicate definitions, e.g. 
$a \wedge b := \min \{ a, b\}$ and $a \vee b := \max \{ a, b\}$. 
For two nonnegative 
functions $f$ and $g$
we write
$f\approx g$ if
there is a positive number $c$, called {\emph constant},
such that $c^{-1}\, g \leq f \leq c\, g$.
	Such comparisons are usually called 
	\emph{sharp estimates}.
We write $c=c(a,b,\ldots,z)$ to claim that $c$ may be so chosen to depend only on $a,b,\ldots,z$.
	For every function $f$, let
$f_+:=f \vee 0$.   For any open subset $D$ of the $d$-dimensional Euclidean space $\R^d$, we denote
by $C^\infty_c (D)$ the space of smooth functions with compact supports in $D$, 
and by $C_c (D)$ the space of continuous  functions with compact supports in $D$. 
In statements and proofs, $c_i$ denote constants whose exact values are unimportant.
These are given anew in each statement and each proof.

\begin{ack}We thank William Beckner for comments on the Hardy equality \eqref{eq:W12}.
We thank Tomasz Byczkowski, Tomasz Grzywny, Tomasz Jakubowski, Kamil Kaleta, Agnieszka Ka\l{}amajska and Dominika Pilarczyk  for comments, suggestions and encouragement. We also thank Rupert L.~Frank and Georgios Psaradakis for remarks on the literature related to Section~\ref{sec:A}.
\end{ack}

\section{
{Non-explosion for Schr\"odinger perturbations}}\label{sec:Sc}
Let $(X,
\mathcal{M}, m(dx))$ be a $\sigma$-finite measure space. Let $\mathcal{B}_{(0,\infty)}$ be the Borel $\sigma$-field on the halfline $(0,\infty)$.
Let $p:(
0,\infty) \times X\times X\to [0,\infty]$ be 
$\mathcal{B}_{(0,\infty)}\times \mathcal{M}\times \mathcal{M}$-measurable and symmetric:
\begin{equation}\label{eq:psymmetric}
p_t(x,y)=p_t(y,x)\,,\quad x,y\in X\,,\quad t>0\,,
\end{equation}
and let $p$ satisfy the Chapman--Kolmogorov equations:
\begin{equation}
  \label{eq:ck}
\int_X p_s(x,y)p_t(y,z)
m(dy)=p_{s+t}(x,z),\qquad 
 x,z\in X,\; s,t> 0,
\end{equation}
 and assume that for every $t>0, x \in X$, 
 $p_t(x,y)m(dy)$
 is ($\sigma$-finite) integral kernel.
Let $f:\R \to [0,\infty)$ be 
non-decreasing, and let
$f=0$ on $(-\infty, 0]$.
We have $f'\geq 0$ a.e., and
\begin{equation}
  \label{eq:fp}
f(a)+\int_a^b f'(s)ds \leq f(b),
\quad -\infty < a\leq b < \infty.
\end{equation}

{Further,} let $\mu$ be a~nonnegative $\sigma$-finite measure on $(X,\mathcal{M})$. We put
\begin{align} 
p_s\mu(x) &= \int_X p_s(x, y) \,\mu(dy), \label{eq:psmu}\\
h(x) &= \int_0^\infty f(s) p_s\mu(x) \,ds. \label{eq:h}
\end{align}
We denote, as usual, $p_th(x)=\int_X p_t(x,y)h(y){m(dy)}$.
By Fubini-Tonelli and Chapman-Kolmogorov, for $t>0$ and $x\in X$ we have 
\begin{align}
\label{eq:kb5.5}
  p_th(x)
  &=\int_t^\infty f(s-t) p_s\mu(x) \,ds\\
&\leq\int_t^\infty f(s) p_s\mu(x) \,ds 
 \nonumber \\
&\leq h(x).\label{eq:ph}
\end{align}
In this sense, $h$ is supermedian.

We define $q:X\to [0,\infty]$ as follows: $q(x)=0$ if $h(x)=0$ or $\infty$, else \begin{equation}\label{eq:qmu}
 q(x) = 
\frac{1}{h(x)} \int_0^\infty f'(s) p_s\mu(x) \,ds.
\end{equation}
For all $x\in X$ we thus have 
\begin{equation}\label{eq:kb8.5}
q(x)h(x)\leq \int_0^\infty f'(s) p_s\mu(x) \,ds.
\end{equation}
We define the Schr\"odinger perturbation of $p$ by $q$ \cite{MR2457489}:
\begin{equation}
  \label{eq:pti}
  \tilde{p}=\sum_{n=0}^\infty p^{(n)},
\end{equation}
where $p^{(0)}_t(x,y)  =  p_t(x,y)$, and
\begin{equation}
\label{eq:pn}
 p^{(n)}_t(x,y)  =  
\int_0^t\int_{X}
p_s(x,z)\,q(z)p^{(n-1)}_{t-s}(z,y)\,m(dz)\,ds,\quad n\geq 1.
\end{equation} 
It is well-known that $\tilde{p}$ is a transition 
density \cite{MR2457489}.
\begin{thm}\label{thm:perturbation}
We have $\int_X \tilde{p}_t(x,y)h(y){m}(dy)\leq h(x)$.
\end{thm}

\begin{proof}
For $n=0,1,\ldots$ and $t>0$, $x\in X$, we consider
\begin{equation*}\label{eq:phn}
p^{(n)}_th(x):=\int_X p^{(n)}_t(x,y)h(y)\,{m(dy)},
\end{equation*}
and we claim that 
\begin{equation}\label{eq:fkb}
 \sum_{k=0}^n p^{(k)}_t h(x) \leq h(x).
\end{equation}
By (\ref{eq:ph}) this holds for $n=0$.
By (\ref{eq:pn}), Fubini-Tonelli, induction and (\ref{eq:qmu}),
\begin{align}
\sum_{k=1}^{n+1}p^{(k)}_t h(x)
&=\int_X \int_0^t \int_X  p_s(x,z)\,q(z)\sum_{k=0}^np^{(k)}_{t-s}(z,y)h(y)\,{m(dy)}\,ds\,{m(dz)}\nonumber\\
&\leq \int_0^t \int_X p_s(x,z)\,q(z)h(z)\,{m(dz)}\,ds \nonumber\\
&=\int_0^t \int_X p_s(x,z)\int_0^\infty f'(u) \int_X p_u(z,w)\,\mu(dw)\,du\,{m(dz)}\,ds.\nonumber \\
&= \int_0^t \int_0^\infty f'(u) p_{s+u}\mu(x) \,du\,ds,\nonumber
\end{align}
where in the last passage we used \eqref{eq:ck} and \eqref{eq:psmu}.
By \eqref{eq:fp},
\begin{align*}
\sum_{k=1}^{n+1}p^{(k)}_t h(x)
&\le  \int_0^\infty \int_0^{u\wedge t} f'(u-s) \,ds\, p_u\mu(x) \,du\\
&\leq  \int_0^\infty [f(u) - f(u - u\wedge t)] \,p_u\mu(x) \,du\\
&=  \int_0^\infty [f(u) - f(u-t)] \,p_u\mu(x) \,du,
\end{align*}
because $f(s)=0$ if $s\le 0$.
By this and 
\eqref{eq:kb5.5} we obtain
\begin{align*}
\sum_{k=0}^{n+1}p^{(k)}_t h(x) &\leq
\int_t^\infty f(u-t) p_u\mu(x) \,du \\
&\quad + \int_0^\infty [f(u) - f(u-t)]\, p_u\mu(x) \,du\\
&= \int_0^\infty f(u) p_u\mu(x)\,du=h(x).
\end{align*}
The claim \eqref{eq:fkb} is proved.
The theorem follows by letting
$n\to \infty$.
\end{proof}
\begin{rem}\label{rem:sl}
Theorem~\ref{thm:perturbation} asserts that $h$ is supermedian for $\tilde p$.
This is much more than 
\eqref{eq:ph}, but \eqref{eq:ph} may also be useful in applications \cite[Lemma~5.2]{2014TJ}.
\end{rem}

We shall see in Section~\ref{sec:A} that the above construction gives integral finiteness (non-explosion) results for {specific} Schr\"odinger perturbations 
with rather singular $q$, cf. 
Corollaries~\ref{cor:aSp} and \ref{cor:aSpL}.
In the next section $q$ will serve as an admissible weight in a Hardy inequality.
\section{
Hardy inequality}\label{sec:Hi}
Throughout this section we let  $p$, $f$, $\mu$, $h$ and $q$ be as defined in Section~\ref{sec:Sc}.
Additionally we shall assume that  $p$ is Markovian, namely $\int_X p_t(x,y){m(dy)}\le 1$ for all $x\in X$.
In short, $p$ is a subprobability transition density.
By Holmgren criterion \cite[Theorem 3, p. 176]{MR1892228}, we then have $p_tu\in L^2(m)$
for each $u\in L^2(m)$, {in fact $\int_X [p_t u(x)]^2m(dx)\le \int_X u(x)^2 m(dx)$}. 
Here $L^2(m)$ is
the collection of all the real-valued square-integrable 
$\mathcal{M}$-measurable functions on $X$. As usual,  we identify $u,v\in L^2(m)$ if $u=v$ $m$-a.e. on $X$. The space (of equivalence classes) $L^2(m)$ is equipped with the scalar product 
$\langle u, v\rangle=\int_X u(x)v(x)m(dx)$. 
Since the semigroup of operators $(p_t,t>0)$ is self-adjoint and weakly measurable, we have
$$
\langle p_t u,u\rangle=\int_{[0,\infty)}e^{-\lambda t}d\langle P_\lambda u,u\rangle,
$$
where $P_\lambda$ is the spectral decomposition of the operators, see \cite[Section~22.3]{MR0423094}.
For $u \in L^2(m)$ and $t>0$ we let
$$ \EEE^{(t)}(u, u)= \frac{1}{t} \langle u - p_t u, u \rangle.$$
{By the spectral decomposition,} 
$t\mapsto \EEE^{(t)}(u, u)$ is nonnegative and nonincreasing \cite[Lemma 1.3.4]{MR2778606}, which allows to define
the quadratic form 
of $p$, 
\begin{align}
\label{def:DEEE}
 \EEE(u,u) &= \lim_{t\to 0}  \EEE^{(t)}(u, u),\quad u\in L^2(m).
\end{align}
The domain of the form is 
defined by the condition $\EEE(u,u)<\infty$
\cite{MR2778606}.

The following is a Hardy-type inequality with a remainder.
\begin{thm}\label{thm:Hardy}
If $u\in L^2(m)$ and $u=0$  on $\{x \in X: h(x)=0 \mbox{ or } \infty\}$, then
\begin{align}
&\EEE(u, u) 
\geq   
  \int_{{X}} u(x)^2 q(x) \,m(dx) \label{eq:Hardy} \\
&+
\liminf_{t\to 0} \int_X \int_X \frac{p_t(x,y)}{2t} \left(\frac{u(x)}{h(x)}-\frac{u(y)}{h(y)} \right)^2 h(y)h(x) m(dy)m(dx).
\nonumber
\end{align}
If  $f(t)=t_+^\beta$ with $\beta\geq 0$ in \eqref{eq:h} or, more generally, if $f$ is 
absolutely continuous and there are $\delta>0$ and $c<\infty$ such that
\begin{equation}\label{eqn:ff}
[f(s) - f(s-t)]/t\le cf'(s) \qquad \mbox{for all $s>0$ and $0<t<\delta$},
\end{equation}
then 
for every $u\in L^2(m)$
\begin{align}
&\EEE(u, u) 
=   
  \int u(x)^2 q(x) \,m(dx) \label{eq:Hardyv} \\
&+
\lim_{t\to 0} \int_X \int_X \frac{p_t(x,y)}{2t} \left(\frac{u(x)}{h(x)}-\frac{u(y)}{h(y)} \right)^2 h(y)h(x) m(dy)m(dx),
\nonumber
\end{align}
\end{thm}
\begin{proof}
Let $v=u/h$, with the convention that $v(x)=0$ if $h(x)=0$ or $\infty$.
Let $t>0$. We note that $|vh|\le |u|$, thus $vh\in L^2(m)$ and
by \eqref{eq:ph},
$v p_th \in L^2(m)$. 
We then have
\begin{align}\label{eq:rEtnJtIt}
 \EEE^{(t)}(vh, vh) 
 &= \langle v \frac{h-p_th}{t}, vh \rangle + 
   \langle  \frac{vp_th-p_t(vh)}{t} , vh \rangle =: I_t + J_t. 
\end{align}
By the definition of $J_t$ and the symmetry \eqref{eq:psymmetric} of $p_t$,  
\begin{align*}
 J_t  &=
 \frac{1}{t} \int_X \int_X p_t(x,y) [v(x)-v(y)]  h(y)\, m(dy)\, v(x) h(x) \,m(dx)  \\
&=
  \int_X \int_X \frac{p_t(x,y)}{2t} [v(x)-v(y)]^2 h(x)h(y) \,m(dx)\,m(dy) \geq 0.
\end{align*}
To deal with $I_t$, we let $x\in X$, assume that $h(x)<\infty$, and consider
\begin{align*}
 (h-p_th)(x) &= \int_0^\infty f(s) p_s\mu(x) \,ds - \int_0^\infty f(s) p_{s+t}\mu(x) \,ds \\
 &=\int_0^\infty [f(s) - f(s-t)] p_s\mu(x) \,ds.
\end{align*}
Thus,
\begin{align*}
I_t&=\int_X v^2(x)h(x) \int_0^\infty \frac1t[f(s) - f(s-t)]\ p_s\mu(x) \,ds\, m(dx).
\end{align*} 
By \eqref{def:DEEE} and Fatou's lemma,
\begin{align}
&\EEE(vh, vh) 
\geq   
  \int_X \int_0^\infty f'(s) p_s\mu(x) \,ds\, v^2(x)h(x)\,m(dx)  \label{eq:prh} \\
&+
\liminf_{t\to 0} \int_X \int_X  \frac{p_t(x,y)}{2t} \left[v(x)-v(y)\right]^2 h(y)h(x) m(dy)m(dx)\nonumber\\
&={ \int_X v^2(x)h^2(x) q(x) m(dx)}\nonumber\\
&{+
\liminf_{t\to 0} \int_X \int_X  \frac{p_t(x,y)}{2t} \left[v(x)-v(y)\right]^2 h(y)h(x) m(dy)m(dx)}.
\nonumber
\end{align}
Since $u=0$  on $\{x \in X: h(x)=0 \mbox{ or } \infty\}$, we have $u=vh$, hence $\EEE^{(t)}(u, u)
=\EEE^{(t)}(vh, vh)$ for all $t>0$, and so $\EEE(u, u)
=\EEE(vh, vh)$. From \eqref{eq:prh} we obtain 
\eqref{eq:Hardy}.

If $f$ is absolutely continuous on $\R$, then \eqref{eq:fp} becomes equality, {and we return to \eqref{eq:rEtnJtIt} to analyse $I_t$ and $J_t$ more carefully.} 
If $\int_X u(x)^2 q(x) \,m(dx)<\infty$, which is satisfied in particular when $\EEE(u,u)<\infty$,
and if \eqref{eqn:ff} holds,
then we can apply
Lebesgue dominated convergence 
theorem to $I_t$. In view of \eqref{def:DEEE} and \eqref{eq:rEtnJtIt}, the limit of $J_t$ then also exists, and we obtain \eqref{eq:Hardyv}.
If $\int_X u(x)^2 q(x) \,m(dx)=\infty$, then \eqref{eq:prh} trivially becomes equality. 
Finally, \eqref{eqn:ff} holds for $f(t)=t_+^\beta$ with $\beta\ge 0$.
\end{proof}
\begin{cor}For every $u\in L^2(m)$ we have $\EEE(u, u) 
\geq   
  \int_{{X}} u(x)^2 q(x) \,m(dx)$.
\end{cor}
We are interested in non-zero quotients $q$. This calls for lower bounds of the numerator and upper bounds of the denominator.
The following consequence of \eqref{eq:Hardy} applies when sharp estimates of $p$ are known.
\begin{cor}\label{rem:cppb}
Assume there are a 
$\mathcal{B}_{(0,\infty)}\times \mathcal{M}\times \mathcal{M}$-measurable 
function 
$\bar p$ and a constant $c\ge 1$ such that for every $(t,x,y) \in (0, \infty)\times X \times X$,
\begin{equation*}\label{eqn:hcomp}
 c^{-1} p_t(x,y) \le  \bar p_t(x,y)  \le c p_t(x,y). \end{equation*}
Let
$$
\bar h(x) = \int_0^\infty  \int_X f(s) \bar p_s(x, y) \mu(dy) \,ds,
$$
and let $\bar q(x)=0$ if $\bar h(x)=0$ or $\infty$, else let
$$
 \bar q(x) = 
\frac{1}{\bar h(x)} \int_0^\infty f'(s) \bar p_s\mu(x) \,ds.
$$
Then $c^{-1}h\le \bar h \le c h$, $c^{-2}q\le \bar q \le c^2 q$, and 
for $u\in L^2(m)$ such that $u=0$  on $X\cap \{\bar h=0 \mbox{ or } \infty\}$, we have
\begin{align}
&\EEE(u, u) 
\geq   
  c^{-2}\int u(x)^2 q(x) \,m(dx). \label{eq:Hardynew} 
\end{align}
\end{cor}
In the remainder of the paper we  discuss applications of the results in Section~\ref{sec:Sc} and Section~\ref{sec:Hi} to transition densities with certain scaling properties.

\section{Applications to (fractional) Laplacian}\label{sec:A}

Let $0<\alpha<2$, $d\in \mathbb{N}$, $\mathcal{A}_{d,-\alpha}=
2^{\alpha}\Gamma\big((d+\alpha)/2\big)\pi^{-d/2}/|\Gamma(-\alpha/2)|$ and
$\nu(x,y)=\mathcal{A}_{d,-\alpha}|y-x|^{-d-\alpha}$, where $x,y\in \Rd$.  Let $m(dx)=dx$, the Lebesgue measure on $\Rd$. 
Throughout this section,  
$g$ is the Gaussian kernel 
\begin{equation}\label{eq:Gk}
 g_t(x) = (4\pi t)^{-d/2} e^{-|x|^2/(4t)}\,, \quad t>0,\quad x\in \Rd\,.
\end{equation}
For  $u\in L^2(\R^d,dx)$ we define
\begin{equation}\label{eqHifL}
\mathcal{E}(u,u)=
\frac12
\int_{\R^d}\!\int_{\R^d} 
[u(x)-u(y)]^2\nu(x,y)
\,dy\,dx.
\end{equation}
The important case $\beta=(d-\alpha)/(2\alpha)$ in the following Hardy equality for the Dirichlet form of the fractional Laplacian was given by 
Frank, Lieb and Seiringer in \cite[Proposition 4.1]{MR2425175}
 (see \cite{MR2984215} for another proof; see also
\cite{MR0436854}).
In fact, \cite[formula (4.3)]{MR2425175} also covers the case of $(d-\alpha)/(2\alpha) \leq \beta \leq (d/\alpha)-1$ and smooth compactly supported functions $u$ in  the following Proposition.
Our proof is  different from that of \cite[Proposition 4.1]{MR2425175} because we do not use the Fourier transform.
\begin{prop}\label{cor:FS}
If $0<\alpha<d\wedge 2$, $0
\le
\beta 
\le(d/\alpha)-1$ and $u\in L^2(\Rd)$, then 
\begin{align*}
&\mathcal{E}(u,u)=
C
\int_{\R^d} \frac{u(x)^2}{|x|^\alpha}\,dx 
 +   \int_{\R^d}\!\int_{\R^d} 
\left[\frac{u(x)}{h(x)}-\frac{u(y)}{h(y)}\right]^2
h(x)h(y) \nu(x,y)
\,dy\,dx\,,
\end{align*}
where $C=2^{\alpha} \Gamma(\frac{d}{2} - \tfrac{\alpha\beta}{2})\Gamma(\frac{\alpha(\beta+1)}{2})\Gamma(\frac{d}{2} - \tfrac{\alpha(\beta+1)}{2})^{-1}\Gamma(\frac{\alpha \beta }{2})^{-1}$,
$h(x)=|x|^{\alpha(\beta+1)-d}$.
We get a maximal $C=2^{\alpha} \Gamma(\tfrac{d+\alpha}{4})^2 \Gamma(\tfrac{d-\alpha}{4})^{-2}$ if  $\beta=(d-\alpha)/(2\alpha)$.
\end{prop}
\begin{proof}
\eqref{eqHifL} is the Dirichlet form of the convolution 
semigroup of functions
defined by subordination, that is we let $p_t(x,y)=p_t(y-x)$, where
\begin{equation}\label{eq:stable-pt}
 p_t(x) = \int_0^\infty g_s(x) \eta_t(s)\,ds,
\end{equation}
$g$ is the Gaussian kernel defined in \eqref{eq:Gk}
and $\eta_t\ge 0$ is the density function of the distribution of the $\alpha/2$-stable subordinator at time $t$, see, e.g., \cite{MR2569321} and \cite{MR2778606}. Thus,
$\eta_t(s) = 0$ for $s\le 0$,
and
\[
 \int_0^\infty e^{-us} \eta_t(s)\,ds = e^{-tu^{\alpha/2}}, \quad u\geq 0.
\]
Let $-1<\beta<d/\alpha-1$. The Laplace transform of $s\mapsto \int_0^\infty t^\beta \eta_t(s)\,dt$ is
\begin{align*}
\int_0^\infty  \int_0^\infty t^\beta \eta_t(s)\,dt\, e^{-us} \,ds 
& =\int_0^\infty  t^\beta \int_0^\infty \eta_t(s) e^{-us} \,ds\,dt 
 =\int_0^\infty  t^\beta e^{-tu^{\alpha/2}} \,dt \\
&= \Gamma(\beta+1) u^{-\frac{\alpha(\beta+1)}{2} }.
\end{align*}
Since $\int_0^\infty e^{-us} s^\gamma \,ds = \Gamma(\gamma+1) u^{-(\gamma+1)}$,
\begin{equation}\label{int:tbetapt}
 \int_0^\infty t^\beta \eta_t(s)\,dt = 
   \frac{\Gamma(\beta+1)}{\Gamma(\frac{\alpha(\beta+1)}{2})} s^{\frac{\alpha(\beta+1)}{2}-1}.
\end{equation}
We consider $-\infty< \delta < d/2 - 1$ and calculate the following integral for the Gaussian kernel by substituting $s=|x|^2/(4t)$,
\begin{align}\label{int:tbetagt}
 \int_0^\infty g_t(x) t^\delta \,dt &= 
  \int_0^\infty (4\pi t)^{-d/2} e^{-|x|^2/(4t)} t^\delta\,dt \\
 &= (4\pi)^{-d/2} \left(\frac{|x|^2}{4}\right)^{\delta-d/2+1} \int_0^\infty s^{d/2-\delta-2} e^{-s}\,ds\nonumber\\
&= 4^{-\delta-1} \pi^{-d/2} |x|^{2\delta-d+2} \Gamma(d/2 - \delta - 1). \nonumber
\end{align}
{For}
$f(t):=t_+^\beta$ and $\sigma$-finite Borel measure $\mu\ge 0$ on $\Rd$ we have
\begin{align*}
 h(x)&:=\int_0^\infty \int_\Rd f(t) p_t(x-y)\mu(dy) \,dt \\
 &=\int_0^\infty \int_\Rd t^\beta \int_0^\infty g_s(x-y)\eta_t(s)\,ds \, \mu(dy) \,dt\\
 &= \int_\Rd \int_0^\infty \int_0^\infty t^\beta\eta_t(s) \,dt\,  g_s(x-y)\,ds \, \mu(dy) \\
&= \int_\Rd \int_0^\infty \frac{\Gamma(\beta+1)}{\Gamma(\frac{\alpha(\beta+1)}{2})} s^{\frac{\alpha(\beta+1)}{2}-1}
g_s(x-y)\,ds \, \mu(dy) \\
&=\frac{\Gamma(\beta+1)}{\Gamma(\frac{\alpha(\beta+1)}{2})} 
   \frac{\Gamma(\frac{d}{2} - \tfrac{\alpha(\beta+1)}{2})}{4^\frac{\alpha(\beta+1)}{2} \pi^{d/2}}
 \int_\Rd |x-y|^{\alpha(\beta+1)-d}\, \mu(dy),
\end{align*}
where  in the last two equalities we assume $\alpha(\beta+1)/2-1 < d/2-1$ and use \eqref{int:tbetapt} and \eqref{int:tbetagt}. If, furthermore, $\beta\ge 0$, then by the same calculation
\begin{align*}
\int_0^\infty \int_\Rd& f'(t) p_t(x,y)\mu(dy) \,dt \\
&=\beta \frac{\Gamma(\beta)}{\Gamma(\frac{\alpha \beta }{2})} 
    4^{-\frac{\alpha\beta}{2}} \pi^{-d/2}  \Gamma(\frac{d}{2} - \tfrac{\alpha\beta}{2})
 \int_\Rd |x-y|^{\alpha\beta-d}\, \mu(dy).
\end{align*}
Here the expression is zero if $\beta=0$. If $\mu=\delta_0$, then we get
\begin{equation}\label{eq:whuL}
h(x)=\frac{\Gamma(\beta+1)}{\Gamma(\frac{\alpha(\beta+1)}{2})} 
   \frac{\Gamma(\frac{d}{2} - \tfrac{\alpha(\beta+1)}{2})}{4^\frac{\alpha(\beta+1)}{2} \pi^{d/2}}
 |x|^{\alpha(\beta+1)-d}
\end{equation}
 and
\begin{align}\label{eq:wquL}
 q(x) &=
 \frac{1}{h(x)}\int_0^\infty \int_\Rd f'(t) p_t(x,y)\mu(dy) \,dt\nonumber\\
& = \frac{ 4^{\alpha/2} \Gamma(\frac{d}{2} - \tfrac{\alpha\beta}{2})\Gamma(\frac{\alpha(\beta+1)}{2}) }
    { \Gamma(\frac{d}{2} - \tfrac{\alpha(\beta+1)}{2})\Gamma(\frac{\alpha \beta }{2}) }  |x|^{-\alpha}.
\end{align}
By homogeneity, we may assume $h(x)=|x|^{\alpha(\beta+1)-d}$, without changing $q$. 
By the second statement of Theorem~\ref{thm:Hardy}, it remains to show that
\begin{align}
\lim_{t\to 0} & \int_{\R^d} \int_{\R^d} \frac{p_t(x,y)}{2t} \left[\frac{u(x)}{h(x)}-\frac{u(y)}{h(y)} \right]^2 h(y)h(x) dydx \nonumber\\
&=\frac12 \int_{\R^d}\!\int_{\R^d} \left[\frac{u(x)}{h(x)}-\frac{u(y)}{h(y)}\right]^2 h(y)h(x)\nu(x,y) \,dy\,dx\,. \label{eq:remains}
\end{align}
Since $p_t(x,y)/t \leq \nu(x,y)$ \cite{MR3165234} and $p_t(x,y)/t \to \nu(x,y)$ as $t\to 0$,
\eqref{eq:remains} follows either by the dominated convergence theorem,
if the right hand side of \eqref{eq:remains} is finite, or -- in the opposite case -- by Fatou's lemma.
If $\alpha\beta = (d-\alpha)/2$, then we obtain $h(x)=|x|^{-(d-\alpha)/2}$ and the maximal
\[
 q(x) =  \frac{ 4^{\alpha/2} \Gamma(\tfrac{d+\alpha}{4})^2 }
    { \Gamma(\tfrac{d-\alpha}{4})^2 }  |x|^{-\alpha}.
\]
Finally, the statement of the proposition is trivial for $\beta=d/\alpha-1$.
\end{proof}

\begin{cor}\label{cor:aSp}
If $0\le r\le d-\alpha$, $x\in \Rd$ and $t>0$, then
$$\int_{\Rd} p_t(y-x)|y|^{-r}dy\le |x|^{-r}.$$
If $0<r< d-\alpha$, $x\in \Rd$, $t>0$, $\beta=(d-\alpha-r)/\alpha$, $q$ is given by \eqref{eq:wquL} and $\tilde p$  is given by \eqref{eq:pti}, then
$$\int_{\Rd} \tilde p_t(y-x)|y|^{-r}dy\le |x|^{-r}.$$
\end{cor}
\begin{proof}
By \eqref{eq:ph} and the proof of Proposition~\ref{cor:FS}, we 
get the first estimate. 
The second estimate is stronger, because $\tilde p\ge p$, cf. \eqref{eq:pti}, and it follows from Theorem~\ref{thm:perturbation}, cf. the proof of Proposition~\ref{cor:FS}.
We do not formulate the second estimate for $r=0$ and $d-\alpha$, because the extension suggested by 
\eqref{eq:wquL} reduces to a special case of the first estimate.
\end{proof}

For completeness we now give Hardy equalities for the Dirichlet form of the Laplacian in $\Rd$.
Namely, \eqref{eq:W12} below is the optimal classical Hardy equality with remainder, 
and \eqref{eq:cHg} is its slight extension, in the spirit of Proposition~\ref{cor:FS}. 
For the equality \eqref{eq:W12}, see for example \cite[formula (2.3)]{MR1918494}, \cite[Section~2.3]{MR2469027} or \cite{MR2984215}. Equality \eqref{eq:cHg} may also be considered as a corollary of \cite[Section~2.3]{MR2469027}.
\begin{prop}\label{cor:W12}
Suppose $d\geq 3$ and $0 \le
\gamma \le d-2$.  For $u\in W^{1,2}(\Rd)$,
\begin{align}\label{eq:cHg}
\int_\Rd |\nabla u(x)|^2 dx \! =\!\gamma (d-2-\gamma) \!\int_\Rd   \frac{u(x)^2}{|x|^{2}}dx+\!\int_\Rd \left|{h(x)}\nabla\frac{u}{h}(x)\right|^2dx,
\end{align}
where $h(x)=|x|^{\gamma+2-d}$.
In particular,
\begin{equation}\label{eq:W12}
 \!\!\int_{\R^d} |\nabla u(x)|^2\,dx 
= \frac{(d-2)^2}{4} \!\!\int_{\R^d}  \frac{u(x)^2}{|x|^2}\,dx+\!\!\int_\Rd \left|{|x|^{\frac{2-d}{2}}}\nabla\frac{u(x)}{|x|^{(2-d)/2}}\right|^2dx.
\end{equation}
\end{prop}
\begin{proof} The first inequality is trivial for $\gamma=d-2$, so let $0\le \gamma<d-2$. We first prove that for $u\in L^2(\R^d,dx)$, 
\begin{align}
\mathcal{C}(u,u)&= \gamma(d-2-\gamma) \int_{\R^d}  \frac{u(x)^2}{|x|^2}\,dx
\nonumber\\
& +\lim_{t\to 0} \int_{\R^d} \int_{\R^d} \frac{{g}_t(x,y)}{2t} \left(\frac{u}{h}(x)-\frac{u}{h}(y)\right)^2 h(y)h(x) dy dx,\label{eq:W12n}
\end{align}
where  $g$ is the Gaussian kernel defined in \eqref{eq:Gk}, and $\mathcal{C}$ is the corresponding quadratic form. 
Even simpler than in the proof of Proposition~\ref{cor:FS},
we let $f(t) = t^{\gamma/2}$ and $\mu=\delta_0$, obtaining
\begin{align}
h(x) &:= \int_0^\infty f(s) g_s \mu(x)\,ds= \int_0^\infty \int_{\R^d} f(s) g_s(x-y) \mu(dy)\,ds \nonumber\\
 &= \int_{\R^d} 4^{-\gamma/2-1} \pi^{-d/2} |x-y|^{\gamma-d+2} \Gamma(d/2 - \gamma/2 - 1) \mu(dy) \nonumber\\
 &=  4^{-\gamma/2-1} \pi^{-d/2} |x|^{\gamma-d+2} \Gamma(d/2 - \gamma/2 - 1),\nonumber\\
 \int_0^\infty f'(s) g_s\mu(x) \,ds &= \frac{\gamma}{2} 4^{-\gamma/2} |x|^{\gamma-d} \pi^{-d/2} \Gamma(d/2 - \gamma/2),\nonumber\\
q(x) &= \frac{\int_0^\infty f'(s) g_s\mu(x) \,ds}{h(x)} = \frac{\gamma(d-2-\gamma)}{|x|^2}.\label{eq:dqL}
\end{align}
By Theorem~\ref{thm:Hardy} we get
\eqref{eq:W12n}. 
Since the quadratic form of the Gaussian semigroup is the classical Dirichlet integral, 
taking $\gamma=(d-2)/2$ and $q(x)=(d-2)^2/(4|x|^2)$ we recover the classical
Hardy inequality:
\begin{equation}\label{eq:W12eq}
 \int_{\R^d} |\nabla u(x)|^2\,dx 
\ge \frac{(d-2)^2}{4} \int_{\R^d}  \frac{u(x)^2}{|x|^2}\,dx,\qquad u\in L^2(\Rd,dx).
\end{equation}
We, however, desire \eqref{eq:cHg}. 
It is cumbersome to directly prove 
the convergence of \eqref{eq:W12n} to \eqref {eq:cHg}\footnote{But see a comment before \cite[(1.6)]{MR2984215} and our conclusion below.}.
Here is an approach based on calculus. For $u\in C_c^\infty(\Rd\setminus\{0\})$ we have
\begin{align*}
&\partial_j  \left(|x|^{d-2-\gamma}u(x)\right)
=
(d-2-\gamma) |x|^{d-4-\gamma}u(x)  x_j + |x|^{d-2-\gamma} u_j(x),\\
&\left|\nabla \left(|x|^{d-2-\gamma}u(x)\right)\right|^2
=
|x|^{2(d-4-\gamma)}   \big[(d-2-\gamma)^2 u(x)^2  |x|^2 +|x|^{4}  |\nabla u(x)|^2  \\
&\qquad\qquad\qquad\qquad\qquad\qquad\qquad\;\;
+ (d-2-\gamma) \left< \nabla (u^2)(x), x\right> |x|^2 \big],
\end{align*}
hence
\begin{align*}
\int_\Rd \left|\nabla\frac{u}{h}(x)\right|^2h(x)^2dx 
&=\int_\Rd  \left|\nabla \left(|x|^{d-2-\gamma}u(x)\right)\right|^2    |x|^{2(\gamma+2-d)}dx   \nonumber \\
&=(d-2-\gamma)^2 \int_\Rd   u(x)^2 |x|^{-2}dx 
 +   \int_\Rd |\nabla u(x)|^2 dx   \nonumber\\
 &\quad  +(d-2-\gamma) \int_\Rd \left< \nabla (u^2)(x), |x|^{-2}x\right> dx   . \label{e:still}
\end{align*}
Since ${\rm div}(|x|^{-2}x)=(d-2) |x|^{-2}$, the divergence theorem yields \eqref{eq:cHg}.
We then
extend \eqref{eq:cHg}
to $u\in C^\infty_c(\Rd)$ as follows.
Let $\psi\in C_c^\infty(\Rd)$ be such that $0\le \psi \le 1$, $\psi(x)=1$ if $|x|\le 1$, $\psi(x)=0$ if $|x|\ge 2$. Let $u_n(x)=u(x)[1-\psi(nx)]$, $n\in \N$. 
We note the local integrability of $|x|^{-2}$ in $\Rd$ with $d\ge 3$.
We let $n\to \infty$ and  
have \eqref{eq:cHg}
hold for $u$ by using the  convergence in $L^2(|x|^{-2}dx)$, 
inequality $|\nabla u_n(x)| \leq |\nabla u(x)| + c_1|u(x)| |x|^{-1}$,
  the identity $h(x)\nabla(u_n/h)(x)=\nabla u_n(x)-u_n(x)[\nabla h(x)]/h(x)$ for $x\neq 0$, the observation that $|\nabla h(x)|/h(x)\leq c|x|^{-2}$ and the dominated convergence theorem. 
We can now extend \eqref{eq:cHg}
to $u\in W^{1,2}(\Rd)$. Indeed, assume that $C^\infty_c(\Rd)\ni v_n\to u$ and 
$\nabla v_n\to g$ in $L^2(\Rd)$
 as $n\to \infty$, so that $g=\nabla u$ in the sense of distributions.
We have that  $h(x)\nabla(v_n/h)(x)=\nabla v_n(x)-v_n(x)[\nabla h(x)]/h(x)\to g-u [\nabla h(x)]/h(x)$ in 
$L^2(\Rd)$. The latter limit is $h\nabla(u/h)$, as we understand it. We obtain
the desired extension of  \eqref{eq:cHg}.
As a byproduct we actually see the convergence of the last term in \eqref{eq:W12n}.
Taking $\gamma=(d-2)/2$ in \eqref{eq:cHg} yields \eqref{eq:W12}.
\end{proof}
We note that \eqref{eq:W12eq} holds for all $u\in L^2(\Rd)$.

\begin{cor}\label{cor:aSpL}
If $0\le r\le d-2$, $x\in \Rd$ and $t>0$, then
$$\int_{\Rd} g_t(y-x)|y|^{-r}dy\le |x|^{-r}.$$
If $0< r< d-2$, $x\in \Rd$, $t>0$, $\beta=(d-2-r)/2$, $q$ is given by \eqref{eq:dqL}, and 
$\tilde g$ is the Schr\"odinger perturbation of $g$ by $q$ as in \eqref{eq:pti}, then
$$\int_{\Rd} \tilde g_t(y-x)|y|^{-r}dy\le |x|^{-r}.$$
\end{cor}
The proof is similar to that of Corollary~\ref{cor:aSp} and is left to the reader.

\section{Applications to transition densities with global scaling}\label{s:asjm}
In this section we show how sharp estimates of transition densities satisfying certain scaling conditions imply Hardy inequalities. 
In particular we give Hardy inequalities for symmetric jump processes on metric measure space studied in  \cite{MR2357678}, and for unimodal L\'evy processes recently estimated in \cite{MR3165234}.
{In what follows we assume that $\phi:[0,\infty)\to [0,\infty)$ is nondecreasing and left-continuous, $\phi(0)=0$, $\phi(x)>0$ if $x>0$ and $\phi(\infty^-):=\lim_{x\to \infty}\phi(x)=\infty$.}
We denote, as usual,
\[
 \phi^{-1}(u) = \inf \{s>0: \phi(s)
{>}u\}, \qquad u
{\geq}0.
\]
{Here is a simple observation, which we give without proof.
\begin{lem}\label{lem:tp}
Let $r,t\ge 0$. We have $t\ge \phi(r)$ if and only if $\phi^{-1}(t)\ge r$.
\end{lem}
}
{We see that $\phi^{-1}$ is upper semicontinuous, hence right-continuous, $\phi^{-1}(\infty^{-})=\infty$,
$\phi(\phi^{-1}(u))\leq u$ and $\phi^{-1}(\phi(s))\geq
s$ for $s,u\ge 0$.
If $\phi$ is continuous, then $\phi(\phi^{-1}(u))= u$, and if $\phi$ is strictly increasing, then 
$\phi^{-1}(\phi(s))= s$ for $s,u\ge 0$.
Both these conditions typically hold in our applications, and then 
$\phi^{-1}$ is the genuine inverse function.
}

We first recall, after \cite[Section 3]{MR3165234}, \cite[Section 2]{MR3131293} and  \cite[(2.7) and (2.20)]{MR2555291}, the following weak scaling conditions.
We say that a function  $\varphi: [0, \infty) \to [0, \infty)$ satisfies the global weak lower scaling condition if there are numbers
$\la \in \R $ and  $\lc \in (0,1] $,  such that
\begin{equation}\label{eq:LSC}
 \varphi(\lambda\theta)\, \ge\,
\lc \,\lambda^{\,\la} \varphi(\theta),\quad \quad \lambda\ge 1, \quad\theta > 0.
\end{equation}We then write $\varphi\in\WLSC(\la, {\lc})$.
Put differently, $\varphi(R)/\varphi(r)\ge \lc \left(R/r\right)^{\la}$, $0<r\le R$.
The global weak upper scaling condition holds if there are numbers $\ua \in \R$
and $\uc \in [1,\infty) $ such that
\begin{equation}\label{eq:USC}
 \varphi(\lambda\theta)\,\le\,
\uc\,\lambda^{\,\ua} \varphi(\theta),\quad\quad \lambda\ge 1, \quad\theta > 0, 
\end{equation}
or $\varphi(R)/\varphi(r)\le \uc \left(R/r\right)^{\ua}$, $0<r\le R$.
In short, $\varphi\in\WUSC(\ua, \uc)$.
We note that $\varphi$ has the lower scaling if and only if $\varphi(\theta)/\theta^\la$
is almost increasing, {i.e. comparable with a nondecreasing function} on $[0,\infty)$, and
$\varphi$ has the upper scaling if and only if $\varphi(\theta)/\theta^\ua$
is almost decreasing, 
see \cite[Lemma~11]{MR3165234}.

Let $(F, \rho, m)$  be a metric measure space with  metric $\rho$ and Radon measure $m$ with full support.
We denote $B(x,r)=\{y\in F: \rho(x,y)<r\}$ and assume that
there is a nondecreasing function $V:[0,\infty)\to [0,\infty)$ such that $V(0)=0$, 
$V(r)>0$ for $r>0$, and 
 \begin{equation}
c_1\, V(r) \leq m (B(x, r)) \leq c_2\, V(r) \quad  \text{for all } x\in F \text{ and  }r
{\geq}
0.
 \label{eqn:univdn}
\end{equation}
We call $V$ the volume function.

\begin{thm}\label{t:hardyphi1}
Let $p$ be a symmetric subprobability transition density on $F$, with Dirichlet form $\EEE$, and assume that
 \begin{align}\label{e:etd}
p_t( x,y)&\approx \frac 1{V(\phi ^{-1}(t))}\wedge
\frac{t}{V(\rho(x,y))\phi (\rho(x,y))},\quad t>0,\quad x, y \in F,
\end{align}
where {
$\phi , V: [0,\infty)\to (0,\infty)$ are non-decreasing, positive on $(0,\infty)$, $\phi(0)=V(0)=0$, $\phi(\infty^-)=\infty$}
 and
$V$ satisfies \eqref{eqn:univdn}. 
If $\lA > \ua>0$, $V\in  \WLSC(\lA, \lC )$ and $\phi \in  \WUSC(\ua, \uc)$,
then there is $C>0$ such that 
\begin{equation}\label{eq:hardyphi1}
 \EEE(u, u)  \ge  C 
   \int_{F} \frac{u(x)^2 }{\phi (\rho(x,y))}\,m(dx),
\qquad
y \in F,\quad u\in L^2(F,m).
\end{equation}
\end{thm}
\begin{proof} 
Let $y \in F$ and $u\in L^2(F,m)$. The constants in the estimates below are independent of $y$ and $u$.
Let $
0<\beta< \lA\ \!/\ua-1$ and define
\begin{equation}\label{e:defh}
h(x)= \int_0^{\infty} 
 t^\beta p_t(x, y)\, dt, \quad x \in F.
\end{equation}
We shall prove that 
 \begin{align}\label{eq:hardyphi11}
&\EEE(u, u)  \approx   
\int_{F} \frac{u(x)^2 }{\phi(\rho(x,y))}\,m(dx) \\
   &+
\liminf_{t\to 0} \int_F \int_F \frac{p_t(x,z)}{2t} \left(\frac{u(x)}{h(x)}-\frac{u(z)}{h(z)} \right)^2 h(z)h(x) m(dz)m(dx).  \nonumber
\end{align}
To this end, we first verify 
\begin{equation}\label{e:ah}
h(x) \approx  \phi (\rho(x, y))^{\beta+1} /V(\rho(x, y)), \qquad \rho(x, y)>0.
\end{equation}
Indeed, 
letting $r=\rho(x, y)>0$
 we 
{first note that Lemma~\ref{lem:tp} yields $t\ge \phi(r)$ equivalent to $tV(\phi^{-1}(t))\ge V(r)\phi(r)$, from whence}
\begin{align*}
 h(x) &\approx 
V(r)^{-1} \phi (r)^{-1} \int_{0}^{\phi (r)} t^{\beta+1}\,dt
+ \int_{\phi (r)}^\infty  \frac{t^\beta}{V(\phi ^{-1}(t))
}\,  dt \\
&=
(\beta+2)^{-1}   \phi (r)^{\beta+1}/V(r) + I.
\end{align*}
To estimate $I$, we observe that 
the assumption $\phi \in  \WUSC(\ua, \uc)$
implies $\phi^{{-1}} \in  \WLSC(1/\ua, \uc^{\ \! -1/\ua})$ \cite[Remark~4]{MR3165234}. 
If $r>0$ and $t\ge \phi(r)$, then
\begin{align*}
\frac{V(\phi^{-1}(t))}{V(r)}&\ge 
\frac{V(\phi^{-1}(t))}{V(\phi^{-1}(\phi(r)))}
\ge 
\lC \left(\frac{\phi^{-1}(t)}{\phi^{-1}(\phi(r))}
\right)^\lA\\
& \ge \lC\ \! \uc^{\ \!-\lA\ \!/\ua} \frac{t^{\ \! \lA\ \!/\ua}}{\phi(r)^{\ \!\lA\ \!/\ua}},
\end{align*}
hence,
\begin{align*}
\frac{t^\beta}{V(\phi ^{-1}(t))
} 
&\le \frac{\uc^{\ \!\lA\ \!/\ua}}{ \lC }  \frac{\phi ({r})^{ \lA\ \!/\ua}\ t^{\beta-\lA\ \!/\ua}}{V(r)}
. \end{align*}
The claim \eqref{e:ah} now follows because 
\begin{align}\label{e:77}
I &\leq 
\frac{\uc^{\ \!\lA\ \!/\ua}}{ \lC }  \frac{\phi ({r})^{ \lA\ \!/\ua}}{V(r)} 
\int_{\phi (r)}^\infty t^{\beta-\lA\ \!/\ua} \,  dt 
= \frac{\uc^{\ \!\lA\ \!/\ua} }{ \lC (\lA\ \!/\ua -1-\beta) } \, \frac{\phi ({r})^{\beta+1}}{ V(r)}.
\end{align}
The function
$$
k(x):= \int_0^{\infty} 
p_t(x,y) (t^\beta)' \,dt, \quad x\in F,
$$
also satisfies
$$
k(x) \approx  \phi (\rho(x, y))^{\beta} /V(\rho(x, y)), \quad x\in F.
$$
This follows by recalculating \eqref{e:ah} for $\beta-1$.
We get
\begin{equation}\label{e:dmp}
C_1 \phi (\rho(x, y))^{-1} \le q(x) :=   \frac{k(x)}{h(x)} \le C_2 \phi (\rho(x, y))^{-1},
\end{equation}
by choosing any 
$\beta \in (0, \lA\ \!/\ua-1)$. 
The theorem follows from \eqref{eq:Hardyv}.
\end{proof}

\begin{rem}\label{rem:ef}
 With the above notation, 
 for each $0<\beta< \lA\ \!/\ua-1$ there exists a constant $c$
such that for all $x,y\in F$ we have
$$\int p_t(x,z)\phi (\rho(z, y))^{\beta+1} /V(\rho(z, y))m(dz)\le c
\phi (\rho(x, y))^{\beta+1} /V(\rho(x, y))
$$
and 
$$
\int \tilde p_t(x,z)\phi (\rho(z, y))^{\beta+1} /V(\rho(z, y))m(dz)\le c
\phi (\rho(x, y))^{\beta+1} /V(\rho(x, y)),
$$
where $\tilde p$ is given by \eqref{eq:pti} with $q(x)=C_1 \phi (\rho(x, y))^{-1}$ on $F$ and
$C_1$
is the constant in the lower bound of the sharp estimate in \eqref{e:dmp}.
 This is a non-explosion result for $\tilde p$, and it is proved in the same way as Corollary~\ref{cor:aSp}.
\end{rem}
\begin{rem}
Interestingly, the Chapman-Kolmogorov equations and \eqref{e:etd} imply that $\phi$ in Theorem~\ref{t:hardyphi1} satisfies a lower scaling, too. We leave the proof of this fact to the interested reader because it is not used in the sequel.  An analogous situation occurs in \cite[Theorem26]{MR3165234}.
\end{rem}

In \cite{MR2357678} a wide class of transition densities are constructed on  
  locally compact separable metric measure spaces $(F, \rho, m)$ with
 metric $\rho$ and Radon measure $m$ of infinite mass and full support on $F$.
Here are some of the assumptions of \cite{MR2357678} (for details see \cite[Theorem 1.2]{MR2357678}).
The functions $\phi, V:[0,\infty)\to(0,\infty)$  are increasing, $\phi(0)=V(0)=0$, $\phi (1)=1$, 
$\phi \in  \WLSC(\la , \lc) \cap \WUSC(\ua, \uc)$,
$V\in  \WLSC(\lA, \lC ) \cap \WUSC(\uA, \uC)$, and
$$\int_0^r\frac {s}{\phi  (s)}\, ds \,\leq \,{c} \, \frac{ r^2}{\phi 
(r)} \quad  \hbox{for every } r>0.
$$
A symmetric measurable function $J(x, y)$ satisfying
\begin{align}\label{eqn:Je}
J(x, y) 
\approx \frac{1}{V(\rho (x, y)) \phi  (\rho (x, y))},\qquad x,y\in F,x\neq y,
\end{align}
is considered in \cite[Theorem~1.2]{MR2357678} along with a symmetric pure-jump Markov process having $J$ as  jump kernel and symmetric $p$ satisfying \eqref{e:etd} as transition density.
By Theorem \ref{t:hardyphi1}, we obtain the following result.

\begin{cor}\label{c:hardyphi1}
Under the assumptions of \cite[Theorem 1.2]{MR2357678}, \eqref{eq:hardyphi1} and \eqref{eq:hardyphi11} hold if $\lA > \ua$. 
\end{cor}

We now specialize to $F=\Rd$ equipped with the Lebesgue measure.
Let $\nu$ be an infinite  isotropic unimodal \emph{L\'evy measure} on $\Rd$ i.e. $\nu(dx)=\nu(|x|)dx$, where $(0,\infty)\ni r\mapsto \nu(r)$ is nonincreasing, and
$$
\nu(\Rd\setminus\{0\})=\infty \quad \text{and} \quad
\int_{\Rd\setminus\{0\}}(|x|^2\wedge 1) \ \nu(dx)<\infty.
$$
Let
\begin{equation}\label{LKfpsi}
\psi(\xi)=\int_{\R^d} \left(1- \cos \left<\xi,x\right>\right) \nu(dx).
\end{equation}
 Because of rotational symmetry,  $\psi$ depends only on
$|\xi|$, and we can write $\psi(r)=\psi(\xi)$ for $r=|\xi|$.
This $\psi$ 
is almost increasing \cite{MR3165234}, {namely $\pi^2\psi(r)\ge \psi^*(r):=\sup\{\psi(p): 0\le p\le r\}$.}  Let $0<\la \leq \ua < 2$.
If  $0 \not\equiv \psi\in\WLSC(\la, \lc)\cap \WUSC(\ua,\uc)$, then
the following defines a convolution semigroup of functions,
\begin{equation}\label{LKf}
p_t(x)=(2\pi)^{-d}\int_{\R^d} e^{i\left<\xi,
{x}\right>}e^{-t\psi(\xi)}d\xi,\quad t>0, x\in 
\Rd,
\end{equation}
and the next two estimates hold \cite[Theorem~21]{MR3165234}. 
\begin{align}
p_t(x)   &\approx
   \left[ \psi^{{-}}(1/t)\right]^{d}\wedge \frac{t\psi(1/|x|)}{|x|^d}  ,\qquad t>0,\; x\in \Rd,  \label{eqn:epnw} \\
 \nu (|x|)   &\approx \frac{\psi(1/|x|)}{|x|^d}
   ,\qquad  x\in \Rd. \label{eqn:epnwj}
\end{align} 
Here $\psi^{-}(u)=\inf\{s\ge 0: \psi^*(s)\ge u\}$, the left-continuous inverse of $\psi^*$.
The corresponding Dirichlet form is
\begin{align*}
  \EEE(u, u) &=(2\pi)^{d}\int_\Rd \hat{u}(\xi)\overline{\hat{v}(\xi)}\psi(\xi)\,d\xi=
\frac12
\int_{\R^d} \int_{\R^d} (u(x)-u(y))^2 \nu(y-x)\,dy\,dx\\
&\approx \int_{\R^d} \int_{\R^d} (u(x)-u(y))^2 \frac{\psi(|x-y|^{-1})}{|y-x|^d} \,dy\,dx,
\end{align*}
cf. \cite[Example~1.4.1]{MR2778606} and the special case discussed in the proof of Proposition~\ref{cor:FS} above.
\begin{cor}\label{cor:hardyphi}
If $d > \ua$, then is $c >0$ such that for all $u\in L^2(\Rd)$
\begin{align}\label{eq:Hu}
\EEE(u, u) & \ge  c  \int_{\R^d} u(x)^2 \psi(1/|x|)\,dx.
\end{align}
\end{cor}
\begin{proof}
Let $0 <\beta < (d /\ua)-1$,
$h(x)=\int_0^{\infty} 
 t^\beta p_t(x)\, dt$, and  $k(x)=\int_0^{\infty} 
 (t^\beta)' p_t(x)\, dt$. 
Considering $\rho(x,y)=|y-x|$, $\phi(r)=1/\psi(1/r)$ and $V(r)=r^d$,
by \eqref{eq:hardyphi1} we get \eqref{eq:Hu} for all $u\in L^2(\Rd)$.
To add some detail, we note that $\phi$ satisfies the same scalings as $\psi^*$ and 
$\phi^{-1}(t)=1/\psi^{-}(t^{-1})$. Thus
\eqref{e:ah} yields $h(x) \approx  \psi(|x|^{-1})^{-\beta-1} |x|^{-d}$ and  $k(x) \approx  \psi(|x|^{-1})^{-\beta} |x|^{-d}$.
In fact, we actually obtain Hardy equality.
Indeed,
\begin{align}
\liminf_{t\to 0} & \int_{\R^d} \int_{\R^d} \frac{p_t(x,y)}{2t} \left[\frac{u(x)}{ h(x)}-\frac{u(y)}{ h(y)} \right]^2  h(y) h(x)dydx \nonumber\\
&=\frac12 \int_{\R^d}\!\int_{\R^d} \left[\frac{u(x)}{ h(x)}-\frac{u(y)}{ h(y)}\right]^2  h(y) h(x)\nu(|x-y|) \,dy\,dx\,,\label{eq:remains1}
\end{align}
because if $t\to 0$, then
 $p_t(x,y)/t \leq c\nu (|x-y|)$ by \eqref{eqn:epnw} and \eqref{eqn:epnwj}, and $p_t(x,y)/t \to \nu (|x-y|)$ (weak convergence of radially monotone functions implies  convergence almost everywhere), and we can use 
the dominated convergence theorem or Fatou's lemma, as in the proof of Proposition~\ref{cor:FS}. 
We thus have a strengthening of \eqref{eq:Hu} for every 
$u\in L^2(\Rd)$:
\begin{align}
\EEE(u, u) =&\int_{\R^d} u(x)^2 \frac{k(x)}{h(x)}\,dx\nonumber\\
&+\frac12 \int_{\R^d}\!\int_{\R^d} \left[\frac{u(x)}{ h(x)}-\frac{u(y)}{ h(y)}\right]^2  h(y) h(x)\nu(|x-y|) \,dy\,dx.
\end{align}
\end{proof}
For instance if we take $\psi(\xi)=|\xi|\sqrt{\log(1+|\xi|)}$, the L\'evy-Kchintchine exponent of a subordinated Brownian motion \cite{MR2978140}, then we obtain 
\begin{equation*}
\int_\Rd |\hat{u}(\xi)|^2 |\xi|\sqrt{\log(1+|\xi|)}d \xi \ge c \int_\Rd \frac{u(x)^2\sqrt{\log(1+|x|^{-1})}}{|x|}dx,\quad u\in L^2(\Rd).
\end{equation*} 
\begin{rem}
We note that
\cite[Theorem 1, the ``thin'' case (T)]{DydaVahakangas} gives \eqref{eq:Hu} for continuous functions $u$ of compact support in $\Rd$. Here we 
extend the result to
all functions $u\in L^2(\Rd)$, as typical for our approach. 
We note in passing that \cite[Theorem~1, Theorem~5]{DydaVahakangas} offers a general framework for Hardy inequalities without the remainder terms and applications for quadratic forms on Euclidean spaces.
\end{rem}

Here is an analogue of Remark~\ref{rem:ef}.
\begin{rem}\label{rem:efL}
Using the notation above, 
for every $0 <\beta < (d -\ua)/\ua$, there exist constants $c_1$, $c_2$
such that  
$$\int p_t(y-x)
\psi(|y|^{-1})^{-\beta-1} |y|^{-d}dy
\le c_1 \psi(|x|^{-1})^{-\beta-1} |x|^{-d}, \quad x\in \Rd,
$$
and 
$$
\int \tilde p_t(x,z)dy)\psi(|y|^{-1})^{-\beta-1} |y|^{-d}dy
\le c_1 \psi(|x|^{-1})^{-\beta-1} |x|^{-d}, \quad x\in \Rd,
$$
where  $\tilde p$ is given by \eqref{eq:pti} with $q(x)=c_2\psi(1/|x|)$ on $\Rd$.
The result is proved as Remark~\ref{rem:ef}. In particular we obtain non-explosion of Schr\"odinger perturbations of such unimodal transition densities with $q(x)=c_2 \psi(1/|x|)$. Naturally, the largest valid $c_2$ is of further interest.
\end{rem}

\section{
Weak local scaling on Euclidean spaces}\label{sec:wls}
In this section we restrict ourselves to the Euclidean space and 
apply Theorem~\ref{thm:Hardy}
to a large class of symmetric jump processes 
satisfying
two-sided heat kernel estimates given
in \cite{MR2806700} and \cite{MR3165234}.
Let $\phi: \R_+\to \R_+$
be a strictly increasing continuous
function such that  $\phi (0)=0$, $\phi(1)=1$, {and}
$$
\lc  \Big(\frac Rr\Big)^{\la } \, \leq \,
\frac{\phi  (R)}{\phi  (r)}  \ \leq \ \uc  \Big(\frac
Rr\Big)^{\ua }
\quad \hbox{for every } 0<r<R \le 1.
$$
Let $J$ be a symmetric measurable function on $\R^d\times \R^d \cap \{x\neq y\}$ 
and let $\kappa_1, \kappa_2$  be positive constants such that
\begin{align}\label{e:bm1}
\frac{\kappa_1^{-1}}{|x-y|^d \phi  (|x-y|)} \le J(x, y) \le  \frac{\kappa_1}{|x-y|^d \phi  (|x-y|)}, \quad |x-y| \le 1,
\end{align}
and
\begin{align}\label{e:bm2}
 \sup_{x\in \R^d} \int_{\{y\in \R^d: |x-y| > 1\}} J(x, y) dy =:\kappa_2 <\infty.
\end{align}
We consider the quadratic form 
$$\EEE(u,u)=\frac12\int_\Rd\int_\Rd [u(y)-u(x)]^2J(x,y)dydx, \qquad u\in L^2(\Rd,dx),$$
with the Lebesgue measure $dx$ as the reference measure, for the symmetric pure-jump Markov processes on $\R^d$ constructed in \cite{MR2524930} from the jump kernel $J(x, y)$.
\begin{thm}\label{cor:hardyphi2}
If $d \ge 3$, then 
\begin{equation}\label{eq:hardyphi2}
 \EEE(u, u) 
 \ge 
 c  \int_{\R^d} u(x)^2 \frac{dx}{\phi (|x|) \vee |x|^2},
\qquad
u\in L^2 (\R^d).
\end{equation}
\end{thm}
\begin{proof}
Let
  ${\mathcal{Q}}$ and $p_t(x,y)$ be the quadratic form and the transition density corresponding to 
the symmetric pure-jump Markov process in $\R^d$ with the
jump kernel 
$J(x, y){\bf 1}_{\{ |x-y| \le 1\}}$ instead of $J(x, y)$, cf.  \cite[Theorem 1.4]{MR2806700}.
Thus,
$$
 {\mathcal{Q}} (u, u)=\frac12
\int_{\R^d\times \R^d} (u(x)-u(y))^2 J(x, y){\bf 1}_{\{ |x-y| \le 1\}} dx dy, \qquad u\in C_c(\R^d).
$$
We define $h$ as 
$
h(x)= \int_0^{\infty} 
p_t(x, 0) t^\beta\, dt$,  $x \in \R^d$
where $
{-1}<\beta < d/2-1$. We note that
for every $T, M\ge 0$,
\begin{align}\label{Gau}
 \int_{T}^{\infty}
 t^{\beta-\frac{d}2} e^{-\frac{Mr^2}{t}} dt
={r^{2\beta-d+2}}  \int_0^{\frac{r^2}{T}}
u^{-2-\beta+\frac{d}2}
 e^{-M u}du.
\end{align}
We shall use \cite[Theorem 1.4]{MR2806700}. We however note that 
the term $\log \frac{|x-y|}{t}$ in the statement of \cite[Theorem 1.4]{MR2806700} should be replaced by $1 + \log_+ \frac{|x-y|}{t}$, to include the case $c^{-1} t \le |x-y| \le c t$ missed in the considerations in \cite{MR2806700}.
With this correction, our arguments are as follows.
When $r =|x|\le 1$,
we have 
\begin{align}\label{e:newq1}
c_0^{-1}\left(\frac 1{\phi^{-1}(t)^d}\wedge \frac{t}{r^d\phi(r)}\right)
\le p_t(x, 0) \le c_0\left(\frac 1{\phi^{-1}(t)^d}\wedge \frac{t}{r^d\phi(r)}\right),  \quad t \in
(0,1]\end{align}
and 
\begin{align}\label{e:newq2} p_t(x, 0)
 \le c_0 \, t^{-d /2}e^{ -\uc
\big( \big(r \, (\log_+ (\frac{r}{t}) +1) \big) \wedge \frac{r^2}t \big)}
 \le c_0 \, t^{-d /2}e^{ -\uc  \frac{r^2}t }, \quad t >1.
\end{align}
Thus, by \eqref{e:newq1}, Lemma \ref{lem:tp}, \eqref{e:newq2},  \eqref{e:77}, \eqref{e:bm1} and \eqref{Gau},
\begin{eqnarray*}
&&c_{3}\frac{\phi (r)^{\beta+1}}{r^d} \le c_0^{-1}
\int_0^{ \phi (r)} \frac{t^{\beta+1}}{r^d\phi (r)}dt  \le h(x)\\
&\le&
 c_0\int_0^{ \phi (r)} \frac{t^{\beta+1}}{r^d\phi (r)}dt + c_0 \int_{ \phi (r)}^{1} \frac {t^{\beta}}{(\phi ^{-1}(t))^d} dt+ c_0 \int_1^{\infty}  t^{\beta-\frac{d}2} e^{-\frac{\uc r^2}{t}} dt\\
&\le& c_4\frac{\phi (r)^{\beta+1}}{r^d}
+ \frac{c_5}{r^{d-2-2\beta}}  \int_0^{\infty}
u^{-2-\beta+d/2}
e^{-\uc  u}du \\
 &\le& c_6
\left({\phi (r)^{1+\beta}}+{r^{2+2\beta}} \right) r^{-d}.
\end{eqnarray*}
If  $r=|x| >  1$, then by \cite[Theorem 1.4]{MR2806700}, 
we have 
\begin{align}\label{e:newq3}
c_0^{-1}  e^{-\lc r \,  (\log_+ (\frac{r}{t}) +1)  }
\le p_t(x, 0)\le c_0  e^{-\uc r \,  (\log_+ (\frac{r}{t}) +1) },  \quad t \in
(0,1],\end{align}
and for $t >1$ we have
\begin{align}\label{e:newq4} c_0^{-1} \, e^{ -\lc  \big( \big(r \,  (\log_+ (\frac{r}{t}) +1)  \big) \wedge \frac{r^2}t \big)} \le 
p_t(x, 0)/t^{-d /2}
 \le c_0 \, e^{ -\uc  \big( \big(r \,  (\log_+ (\frac{r}{t}) +1)  \big) \wedge \frac{r^2}t \big)}.
\end{align}
In particular, 
\begin{align}\label{e:newq5}  
p_t(x, 0)
 \ge c_0^{-1} \, t^{-d /2}e^{ -\lc  \big( \big(r \,  (\log_+ (\frac{r}{t}) +1)  \big) \wedge \frac{r^2}t \big)}
 \ge c_7 t^{-d /2}, \quad t >r^2.
\end{align}
Then \eqref{e:newq3}, \eqref{e:newq4},  \eqref{e:newq5},
\eqref{e:bm1} and \eqref{Gau} give
\begin{eqnarray*}
&&\frac{c_{7}}{r^{d-2-2\beta}}  \int_0^{1}
u^{-2-\beta+d/2}
du = c_{7} \int_{r^2}^{\infty}t^{\beta-\frac{d}2}  dt
 \le h(x)\\
&\le& c_{8}\int_0^{r} t^{\beta} e^{-c_{9}r}     dt
+
c_{8} \int_{r}^{\infty}  t^{\beta-\frac{d}2} e^{-\frac{c_{10} r^2}{t}} dt 
\\
 &=&   c_{8} (\beta+1)^{-1}r^{\beta+1} e^{-c_{9} r} + \frac{c_{8}}{r^{d-2-2\beta}}   \int_0^{r}
u^{-2-\beta+d/2}e^{-c_{10} u}du\\
 &\le&  \frac{c_{11}}{r^{d-2-2\beta}}.
\end{eqnarray*}
Thus,
\[
h(x) \approx  (\phi (|x|) \vee |x|^2)^{\beta+1}  |x|^{-d}, \quad x\in \R^d.
\]
In particular, if $
{0}<\beta <d/2-1$, and
$k(x)= \int_0^{\infty} 
p_t(x,0) (t^\beta)' \,dt$,
then
\[
k(x) \approx  (\phi (|x|) \vee |x|^2)^{\beta}|x|^{-d}, \quad x\in \R^d.
\]
Therefore,
\[
q(x) :=   \frac{k(x)}{h(x)} \approx  \frac1{\phi (|x|) \vee |x|^2}.
\]
Theorem~\ref{thm:Hardy} 
yields
\begin{equation}\label{eq:hardyphi2p}
\EEE(u, u)  \ge {\mathcal{Q}}(u, u) \ge 
 c_{12}  \int_{\R^d} u(x)^2 \frac{dx}{\phi (|x|) \vee |x|^2},
\quad
u\in L^2(\R^d).
\end{equation}
\end{proof}
\begin{rem}\label{rem:efLl}
As in Remark~\ref{rem:efL} we obtain non-explosion for Schr\"odinger perturbations by $q(x)=
c/[\phi (|x|) \vee |x|^2]$.
\end{rem}
\begin{rem}
The arguments and conclusions of Theorem~\ref{cor:hardyphi2} are valid for the unimodal L\'evy processes, in particular for the subordinated Brownian motions, provided their L\'evy-Khintchine exponent $\psi$ satisfies the assumptions of local scaling conditions at infinity with exponents strictly {between $0$ and $2<d$} made in \cite[Theorem~21]{MR3165234}:
\begin{align*}\label{eq:Hu1e}
\int_\Rd |\hat{u}(\xi)|^2 \psi(\xi) d \xi &\ge c  \int_{\R^d} u(x)^2 \left[\psi\left(\frac1{|x|}\right)\wedge \frac1{|x|^2}\right]\,dx, \qquad u\in L^2(\Rd).
\end{align*}

\end{rem}


\begin{thebibliography}{10}

\bibitem{MR856511}
A.~Ancona.
\newblock On strong barriers and an inequality of {H}ardy for domains in {${\bf
  R}\sp n$}.
\newblock {\em J. London Math. Soc. (2)}, 34(2):274--290, 1986.

\bibitem{MR742415}
P.~Baras and J.~A. Goldstein.
\newblock The heat equation with a singular potential.
\newblock {\em Trans. Amer. Math. Soc.}, 284(1):121--139, 1984.

\bibitem{MR2984215}
W.~Beckner.
\newblock Pitt's inequality and the fractional {L}aplacian: sharp error
  estimates.
\newblock {\em Forum Math.}, 24(1):177--209, 2012.

\bibitem{MR2569321}
K.~Bogdan, T.~Byczkowski, T.~Kulczycki, M.~Ryznar, R.~Song, and
  Z.~Vondra{\v{c}}ek.
\newblock {\em Potential analysis of stable processes and its extensions},
  volume 1980 of {\em Lecture Notes in Mathematics}.
\newblock Springer-Verlag, Berlin, 2009.
\newblock Edited by Piotr Graczyk and Andrzej Stos.

\bibitem{MR2663757}
K.~Bogdan and B.~Dyda.
\newblock The best constant in a fractional {H}ardy inequality.
\newblock {\em Math. Nachr.}, 284(5-6):629--638, 2011.

\bibitem{MR3165234}
K.~Bogdan, T.~Grzywny, and M.~Ryznar.
\newblock Density and tails of unimodal convolution semigroups.
\newblock {\em J. Funct. Anal.}, 266(6):3543--3571, 2014.

\bibitem{MR2457489}
K.~Bogdan, W.~Hansen, and T.~Jakubowski.
\newblock Time-dependent {S}chr\"odinger perturbations of transition densities.
\newblock {\em Studia Math.}, 189(3):235--254, 2008.

\bibitem{MR2524930}
Z.-Q. Chen, P.~Kim, and T.~Kumagai.
\newblock On heat kernel estimates and parabolic {H}arnack inequality for jump
  processes on metric measure spaces.
\newblock {\em Acta Math. Sin. (Engl. Ser.)}, 25(7):1067--1086, 2009.

\bibitem{MR2806700}
Z.-Q. Chen, P.~Kim, and T.~Kumagai.
\newblock Global heat kernel estimates for symmetric jump processes.
\newblock {\em Trans. Amer. Math. Soc.}, 363(9):5021--5055, 2011.

\bibitem{MR2357678}
Z.-Q. Chen and T.~Kumagai.
\newblock Heat kernel estimates for jump processes of mixed types on metric
  measure spaces.
\newblock {\em Probab. Theory Related Fields}, 140(1-2):277--317, 2008.

\bibitem{Dyda12}
B.~Dyda.
\newblock Fractional calculus for power functions and eigenvalues of the
  fractional {L}aplacian.
\newblock {\em Fract. Calc. Appl. Anal.}, 15(4):536--555, 2012.

\bibitem{DydaVahakangas}
B.~{Dyda} and A.~V. {V\"ah\"akangas}.
\newblock {A framework for fractional Hardy inequalities.}
\newblock {\em {Ann. Acad. Sci. Fenn., Math.}}, 39(2):675--689, 2014.

\bibitem{MR1918494}
S.~Filippas and A.~Tertikas.
\newblock Optimizing improved {H}ardy inequalities.
\newblock {\em J. Funct. Anal.}, 192(1):186--233, 2002.

\bibitem{MR1786080}
P.~J. Fitzsimmons.
\newblock Hardy's inequality for {D}irichlet forms.
\newblock {\em J. Math. Anal. Appl.}, 250(2):548--560, 2000.

\bibitem{MR2425175}
R.~L. Frank, E.~H.~Lieb, and R.~Seiringer.
\newblock Hardy-{L}ieb-{T}hirring inequalities for fractional {S}chr\"odinger operators
\newblock {\em J. Amer. Math. Soc.}, 21(4):925--950, 2008.

\bibitem{MR2469027}
R.~L. Frank and R.~Seiringer.
\newblock Non-linear ground state representations and sharp {H}ardy
  inequalities.
\newblock {\em J. Funct. Anal.}, 255(12):3407--3430, 2008.

\bibitem{MR2778606}
M.~Fukushima, Y.~Oshima, and M.~Takeda.
\newblock {\em Dirichlet forms and symmetric {M}arkov processes}, volume~19 of
  {\em de Gruyter Studies in Mathematics}.
\newblock Walter de Gruyter \& Co., Berlin, extended edition, 2011.

\bibitem{MR0436854}
I.~W. Herbst.
\newblock Spectral theory of the operator
  {$(p\sp{2}+m\sp{2})\sp{1/2}-Ze\sp{2}/r$}.
\newblock {\em Comm. Math. Phys.}, 53(3):285--294, 1977.

\bibitem{MR0423094}
E.~Hille and R.~S. Phillips.
\newblock {\em Functional analysis and semi-groups}.
\newblock American Mathematical Society, Providence, R. I., 1974.
\newblock Third printing of the revised edition of 1957, American Mathematical
  Society Colloquium Publications, Vol. XXXI.

\bibitem{2014TJ}
T.~Jakubowski.
\newblock Fundamental solution of the fractional diffusion equation with a
  singular drift.
\newblock preprint, 2014.

\bibitem{MR2852869}
A.~Ka{\l}amajska and K.~Pietruska-Pa{\l}uba.
\newblock On a variant of the {G}agliardo-{N}irenberg inequality deduced from
  the {H}ardy inequality.
\newblock {\em Bull. Pol. Acad. Sci. Math.}, 59(2):133--149, 2011.

\bibitem{MR3131293}
P.~Kim, R.~Song, and Z.~Vondra{\v{c}}ek.
\newblock Global uniform boundary {H}arnack principle with explicit decay rate
  and its application.
\newblock {\em Stochastic Process. Appl.}, 124(1):235--267, 2014.

\bibitem{MR1892228}
P.~D. Lax.
\newblock {\em Functional analysis}.
\newblock Pure and Applied Mathematics (New York). Wiley-Interscience [John
  Wiley \& Sons], New York, 2002.

\bibitem{MR2777530}
V.~Maz'ya.
\newblock {\em Sobolev spaces with applications to elliptic partial
  differential equations}, volume 342 of {\em Grundlehren der Mathematischen
  Wissenschaften [Fundamental Principles of Mathematical Sciences]}.
\newblock Springer, Heidelberg, augmented edition, 2011.

\bibitem{MR3020137}
D.~Pilarczyk.
\newblock Self-similar asymptotics of solutions to heat equation with inverse
  square potential.
\newblock {\em J. Evol. Equ.}, 13(1):69--87, 2013.

\bibitem{MR2978140}
R.~L. Schilling, R.~Song, and Z.~Vondra{\v{c}}ek.
\newblock {\em Bernstein functions}, volume~37 of {\em de Gruyter Studies in
  Mathematics}.
\newblock Walter de Gruyter \& Co., Berlin, second edition, 2012.
\newblock Theory and applications.

\bibitem{MR3117146}
I.~Skrzypczak.
\newblock Hardy-type inequalities derived from {$p$}-harmonic problems.
\newblock {\em Nonlinear Anal.}, 93:30--50, 2013.

\bibitem{MR1760280}
J.~L. Vazquez and E.~Zuazua.
\newblock The {H}ardy inequality and the asymptotic behaviour of the heat
  equation with an inverse-square potential.
\newblock {\em J. Funct. Anal.}, 173(1):103--153, 2000.

\bibitem{MR2555291}
M.~Z{\"a}hle.
\newblock Potential spaces and traces of {L}\'evy processes on {$h$}-sets.
\newblock {\em Izv. Nats. Akad. Nauk Armenii Mat.}, 44(2):67--100, 2009.

\end{thebibliography}

\end{document}